\documentclass[12pt, a4paper]{amsproc}

\usepackage[
unicode=true,
plainpages = false,
pdfpagelabels,
bookmarks=true,
bookmarksnumbered=true,
bookmarksopen=true,
breaklinks=true,
backref=false,
colorlinks=true,
linkcolor = blue,
urlcolor  = blue,
citecolor = red,
anchorcolor = green,
hyperindex = true,
hyperfigures
]{hyperref}
\hypersetup{
	pdftitle={Linear algebra, theory and algorithms},
	pdfauthor={V.H. Mikaelian},
	pdfsubject={Linear Algebra}
}

\newcommand{\sub}{\subseteq}

\newcommand{\V}{{\mathfrak V}}
\newcommand{\Ni}{{\mathfrak N}}

\renewcommand{\U}{{\mathfrak U}}
\newcommand{\A}{{\mathfrak A}}
\newcommand{\D}{{\mathfrak D}}
\newcommand{\X}{{\mathfrak X}}
\newcommand{\Y}{{\mathfrak Y}}
\newcommand{\B}{{\mathfrak B}}
\newcommand{\N}{{\mathbb N}}
\newcommand{\Z}{{\mathbb Z}}

\newcommand{\F}{{\mathbb F}}

\newcommand{\1}{\{1\}}

\newcommand{\Aut}[1]{\mathrm{Aut}\left( #1 \right)}

\newcommand{\var}[1]{\mathrm{var}\left( #1 \right)}

\newcommand{\varr}[1]{\mathrm{var}\,  #1 }

\newcommand{\Wr}{\,\mathrm{Wr}\,}
\newcommand{\Wrr}{\,\mathrm{wr}\,}

\newcommand{\abs}[1]{\lvert#1\rvert}

\newtheorem{theorem}{Theorem}[section]
\newtheorem{lemma}[theorem]{Lemma}
\newtheorem{corollary}[theorem]{Corollary}  
\newtheorem{problem}[theorem]{Problem} 
\theoremstyle{definition}
\newtheorem{definition}[theorem]{Definition}
\newtheorem{example}[theorem]{Example}

\theoremstyle{remark}
\newtheorem{remark}[theorem]{Remark}

\numberwithin{equation}{section}

\begin{document}

{\parindent=0mm \footnotesize
Proc.~Conf.~Abelian Groups, Rings and Modules \\
Perth, Australia, 9-15, July, 2000\\
{\bf Contemprorary Mathematics, Amer. Math. Soc.}, 273, Providence, RI (2001), 223-238.\\ 
ISSN 0271--4132,\;
MR1817165,\;
Zbl 0981.20018,\;
\href{http://www.ams.org/books/conm/273/}{AMS books},\;  
DOI: dx.doi.org/10.1090/conm/273 
\vskip5mm
}
\vskip5mm

\title[On varieties of groups generated by wreath products]{On varieties of groups generated by 
wreath products of abelian groups}
\author{Vahagn H.~Mikaelian}
\email{v.mikaelian@gmail.com}

\renewcommand{\subjclassname}{2010 Mathematics Subject Classification}

\date{May 17, 2000.}

\dedicatory{To Marine Mikaelian on her birthday}

\begin{abstract}
Generalizing results of Higman and  Houghton on varieties generated by wreath products of finite cycles, we prove that the (direct or cartesian) wreath product of arbitrary abelian groups $A$ and $B$ generates the product variety $\var A \cdot \var B$ if and only if one of the groups $A$ and $B$ is not of finite exponent, or if  $A$ and $B$ are of finite exponents $m$ and $n$ respectively and for all primes $p$ dividing both $m$ and $n$, the factors $B[p^k]/B[p^{k-1}]$ are infinite, where $B[s]=\langle b\in B|\,b^{s}=1 \rangle$ and where $p^k$ is the highest power of $p$ dividing~$n$.
\end{abstract}

\maketitle

\section*{Introduction}

\noindent
The problem,  whether the standard wreath product $A \Wrr B$ of abelian groups  $A$ and $B$
 generates the product variety $ \var A \cdot \var B$, is solved by Higman for the case when
$A=C_p$ and $B=C_n$ are finite cycles of  orders $p$ and $n$, where $p$ is a
prime~\cite{Some_remarks_on_varieties}, and by Houghton for the case of arbitrary finite cycles
$A=C_m$ and $B=C_n$. Namely, equality $\var{C_m \Wrr C_n}= \var{C_m} \cdot \var{C_n} = \A_m
\cdot \A_n$ holds if and only if $m$ and $n$ are coprime. 
As we are informed by Professor C.~Houghton, his result never was published. However this theorem is frequently cited in the literature and can be found, say, in~\cite{HannaNeumann}: this is not only a well-known result of
independent interest, but also an argument frequently used in  other constructions of the theory of varieties of groups: descriptions of lattices of subvarieties of certain product varieties, basis ranks of varieties (see for example~\cite{HannaNeumann}).

The aim of this paper is to generalize Houghton's result for the case of {\it arbitrary} 
abelian groups $A$ and $B$. Namely, {\it for arbitrary abelian groups $A$ and $B$ the (direct or
cartesian) wreath product of groups $A$ and $B$ generates the product variety $\var A \cdot
\var B$ if and only if at least one of the groups $A$ and $B$ is not of finite exponent, or if 
$A$ and $B$ are of finite exponents $m$ and $n$ respectively and for all primes $p$ dividing
both $m$ and $n$, the factors $B[p^k]/B[p^{k-1}]$ are infinite, where $B[s]$ is defined as
$B[s]=\langle b\in B|\,b^{s}=1 \rangle$ and where $p^k$ is the highest power of $p$
dividing~$n$} (Theorem~\ref{MainCriterion}). As we will see below these factors
$B[p^k]/B[p^{k-1}]$ have very ``understandable'' structure and our criterion is easily
applicable in concrete situations. 
 
The structure of infinitely generated abelian groups of non-finite exponent is complicated and  at first sight such a generalization may demand techniques very different from those of critical groups, of Cross varieties or of finite nilpotent $p$-groups. The main idea that enables us to deal with the case of infinitely generated groups is the following main dichotomy: {\it each abelian group is either of finite exponent and, thus, is a direct sum of (possibly infinitely many) copies of some finitely many cycles of prime power orders, or is a discriminating group} (see definitions and notations below). The point  is that if the ``active'' group $B$ is a discriminating group, then the cartesian or direct wreath product of $A$
and $B$ always generates (and discriminates) the variety $\var A \cdot \var B$ and we can restrict ourselves to the first case of the dichotomy. In this case if $p$ is a prime dividing $n$, then the $p$-primary component $B_p$ of $B$ is simply a direct sum of some (possibly infinitely many) copies of cycles $C_p, C_{p^2}, \ldots ,C_{p^k}$, and our condition $|B[p^k]/B[p^{k-1}]|=\infty$ simply means that the direct decomposition of $B_p$ contains infinitely many summands isoporphic to the
cycle $C_{p^k}$.

Since the proof of Theorem~\ref{MainCriterion} consists of consideration of several cases and 
subcases, we have divided it into parts which occupy
Sections~\ref{The_cases_of_non-finite_exponent}--%
\ref{groups_of_finite_composite_exponents}
and each one of them is presented as an independent result 
(Theorems~\ref{InfiniteExpAorB},~\ref{FinitelyGenerated}, 
~\ref{p-groups},~\ref{ArbitraryFinitem,n} closing corresponding sections).

For arbitrary groups $A$ and $B$ the cartesian wreath product $A\Wr B$ and direct wreath product $A\Wrr B$ generate the very same variety of groups. We build our construction for the case of {\it cartesian}\, wreath products,  bearing in mind, that our proofs are also true for {\it direct} wreath products of groups. Only in a few cases do we consider direct wreath products for some specific details of the proofs. 

For general information on the theory of groups we refer 
to~\cite{Robinson,KargapolovMerzlyakov}. Following~\cite{HannaNeumann,OnTheStructureOfWreathProducts} we denote by $A \Wr B$ the {\it cartesian}\, wreath product of groups $A$ and $B$ and by $A \Wrr B$ the {\it direct}\, wreath product of these groups.  The {\it base groups}\, of the cartesian and direct wreath products of $A$ and $B$ will be denoted by  $A^B$ and by $A^{(B)}$ respectively. Detail information on
wreath products can be found in 
~\cite{HannaNeumann,MeldrumWr,OnTheStructureOfWreathProducts,KargapolovMerzlyakov}.
For general information on varieties of groups we refer to
the book of Hanna Neumann~\cite{HannaNeumann}. We reserve notations $\A$,  $\A_n$, $\Ni_c$ and $\B_e$ for varieties of all abelian groups, of all abelian  groups of exponent dividing $n$, of all nilpotent groups of class at most $c$, and of all groups of exponents dividing $e$ respectively. For a  set $\X$ of groups we denote by $\varr{\X}$, as usual, the variety {\it generated} by $\X$. Information on the notion of {\it discriminating group} can be found in~\cite{B+3N,HannaNeumann}. See also the articles of Bryce~\cite{Bryce_Metabelian_groups,Bryce_Variaties_Metabelian_p-groups} and of Kov\'acs and Newman~\cite{KovacsAndNewmanOnInfiniteGroups} for results of more general nature related to the material of this paper. 
We write abelian groups {\it additively}, all other groups will be written {\it multiplicatively}. Background information on abelian groups used in this paper can be found in~\cite{Fuchs_Infinite_groups,Robinson,KargapolovMerzlyakov}.

I am extremely grateful to Professor Alexander Yurievich Ol'shanskii, who introduced me to varieties of groups and guided and encouraged me in all parts of my work during my post-graduate study at the Lomonosov  Moscow State University, where most of this investigation was done.

\section{Wreath products and operations ${\sf Q,S,C}$}  
\label{QSC}

\noindent
As usual, for a given set $\X$ of groups we denote by  ${\sf Q}\X$, ${\sf S}\X$ and  ${\sf C}\X$, 
the sets of all homomorphic images, subgroups and cartesian products of groups of $\X$
respectively. According to Birkhoff's Theorem~\cite{BirkhoffQSC,HannaNeumann}, for the given
set $\X$ of groups the variety  $\varr{\X}$ can be realized as: $\varr{\X}={\sf QSC}\,\X$.

For given $\X$ and $\Y$ denote $\X \Wr \Y = \{ X\Wr Y \,|\, X\in \X, Y\in \Y\}$ and  $\X \Wrr \Y = \{ X\Wrr Y \,|\, X\in \X, Y\in \Y\}$. Since the product variety $\varr{\X} \cdot \varr{\Y}$ consists of all extensions of groups $X^*\in \varr{\X}$ by groups $Y^*\in \varr{\Y}$
and, since by Kaloujnine and Krasner Theorem~\cite{KaloujnineKrasner} each extension of such a type can be embedded into the appropriate wreath product $X^* \Wr Y^*$, we get that the set $\X \Wr \Y $ generates the variety $\varr{\X}\cdot \varr{\Y}$ if and only if for each pair $X^* \in
{\sf QSC}\,{\X}$ and $Y^* \in {\sf QSC}{\Y}$  variety $\var{\X \Wr \Y}$ contains $X^* \Wr Y^*$

The following lemmas, however, show that, to see whether $\var{\X \Wr \Y}=\varr{\X}\cdot 
\varr{\Y}$, for purposes of the current paper we have to check just {\it one} of six conditions
assumed, namely, whether for abelian sets $\X$ and $\Y$ of groups the variety $\var{\X \Wr \Y}$
contains wreath products $X \Wr Y^*$ for every $X \in \X$ and $Y^* \in {\sf C}\Y$.

\begin{lemma}
\label{X*WrY_belongs_var}
For arbitrary sets $\X$ and $\Y$ of groups and arbitrary groups $X^*$ and $Y$, where $X^*\in 
{\sf Q}\X$, $X^*\in {\sf S}\X$ or $X^*\in {\sf C}\X$ and where $Y\in \Y$, the group $X^* \Wr
Y$  belongs to variety $\var{\X \Wr \Y}$.
\end{lemma}

\begin{proof}
If $X^*$ is the homomorphic image of some $X\in \X$ under a homomorphism $f$, then for arbitrary 
$Y\in \Y$ the group  $X^* \Wr Y$ is the homomorphic image of the group $X \Wr Y$ under the
homomorphism  $f_W$ defined as: $f_W\!\!:y\varphi \mapsto y\varphi_f$, where $y\in Y$, $\varphi \in
X^Y$ and where $\varphi_f \in (X^*)^Y$ is set as: $\varphi_f(g)=f(\varphi (g))$ for each $g\in
Y$~\cite[22.11]{HannaNeumann}.

If $X^*$ is the subgroup of some $X\in \X$, then, clearly, $X^* \Wr Y$ is the subgroup of $X 
\Wr Y$~\cite[22.12]{HannaNeumann}.

If $X^*$ is cartesian product of some groups $X_i\in \X$, $i\in I$, then we can define an 
embedding of $X^* \Wr Y$ into the cartesian product $W^*=\prod_{i\in I}(X_i \Wr Y)$ by the
following rule:
$
	y \varphi^* \mapsto \theta \in W^*,
$
where $y\in Y$, $\varphi^* \in (X^*)^Y$ and $\theta$ it defined as $\theta(i)=y \varphi_i$ with 
$\varphi_i \in X_i^Y$, $\varphi_i (g)=[\varphi (g)](i)$, $g\in Y$, $i\in I$.
\end{proof}

\begin{lemma}
\label{X*WrY_belongs_var}
For arbitrary sets $\X$ and $\Y$ of groups and arbitrary groups $X$ and $Y^*$, where $X\in\X$ and 
where $Y^*\in {\sf S}\Y$, the group $X \Wr Y^*$  belongs to variety $\var{\X \Wr \Y}$.
Moreover, if  $\X$ is a set of abelian groups, then for each  $Y^*\in {\sf Q}\Y$ the group $X
\Wr Y^*$  also belongs to $\var{\X \Wr \Y}$.
\end{lemma}

\begin{proof}
The first statement of the lemma is obvious (see~\cite[22.13]{HannaNeumann}).

Since the cartesian and direct wreath products $H\Wr G$ and $H\Wrr G$ of arbitrary groups $H$ and 
$G$ generate the same variety~\cite[22.31, 22.32]{HannaNeumann}, it is sufficient to show that,
if $Y^*$ is a homomorphic image of some $Y\in \Y$ under some homomorphism $h$, then the direct
wreath product $X\Wrr Y^*$ is the homomorphic image of $X\Wrr Y$ under some homomorphism $h_W$,
provided that, $X$ is abelian. $h_W$ is defined by its values $h_W(y)$ and $h_W(\varphi_{x,y})$
over the following set of elements generating $X\Wrr Y$:
$$
	\{ y\in Y, \,\varphi_{x,y}\in X^{(Y)} \,|\,
	\varphi_{x,y}(y)=x \,\,\, {\rm and} \,\,\, 
	\varphi_{x,y}(g)=1 \,\,\, {\rm for} \,\,\, 
	g\in Y \backslash \{ y \},\,\, x\in X \}.
$$
Namely:
$$
	h_W\!\!:y \mapsto h(y) \quad {\rm and} \quad
	h_W\!\!:\varphi_{x,y}\mapsto \varphi_{x,h(y)}\in X^{(Y^*)}.
$$
See also~\cite{ShmelkinOnCrossVarieties,BrumbergOnWreathProducts}.
\end{proof}

In particular, if each of $\X$ and $\Y$ consist of one group only, it follows from the previous two lemmas that:

\begin{lemma}
\label{karevor}
For arbitrary groups $A$ and $B$, if $A^*\cong A/N$ ($N$ is any normal subgroup of $A$), $A^* 
\le A$ or $A^*=\prod_{i\in I}A$ ($I$ is any index set), then $A^*\Wr B \in \var{A\Wr B}$.

On the other hand, if $B^*\cong B/K$ ($A$ is abelian and $K$ is any normal subgroup of $B$) or if $B^* \le B$, then $A\Wr B^* \in \var{A\Wr B}$.
\end{lemma}

\begin{corollary}
If $A$ and $B$ are abelian groups, then $\var{A\Wr B}=\var{A}\cdot\var{B}$ if and only if 
$\var{A\Wr B}$ contains the wreath product $A\Wr (\prod_{i\in I}B)$ for every index set $I$.
\end{corollary}

\begin{remark}
As we will see in Section~\ref{last}, an even stronger result of independent interest can be proved: $\var{A\Wr B}=\var{A}\cdot\var{B}$ holds for abelian groups $A$ and $B$ if and only if $\var{A\Wr B}$ contains the wreath product $A\Wr(B\oplus B)$ of $A$ and of the direct sum of {\it two copies} of $B$ (see Theorem~\ref{AnEvenStrongerResult}).
\end{remark}

\section{Discriminating sets of groups\\ and the case of abelian groups of non-finite exponents}  
\label{The_cases_of_non-finite_exponent}

\noindent
Let us begin by considering the case of wreath products of abelian groups $A$ and $B$, 
where at least one of these groups is {\it not} of finite exponent. As we will see, this
situation can be described via properties of discriminating sets of groups.

\begin{definition}[see~\cite{B+3N}]
The set $\D$ of groups is said to be {\it discriminating}, if for arbitrary finite word set $V$ with the property that, for each $w\in V$ there exists a homomorphism $\delta_w$ of a free group
$F_n$ into some group of $\D$, such that $\delta_w(w)\not= 1$, there exist a group $D\in \D$ and a  single homomorphism $\delta$ of $F_n$ into $D$, such that $\delta(w)\not= 1$ for all $w\in V$.
\end{definition}

A discriminating set of groups $\D$ can be described by the following property: every finite set of identities $\{w\equiv 1 \,|\, w\in V \}$ that can be separately falsified in some groups
$\{D_w \in \D\,|\, w\in V\}$ can also be {\it simultaneously} falsified in a group $D=D_V \in
\D$ for certain choice of values $d_1,d_2,\ldots,d_n\in D$. Every discriminating set $\D$ {\it
discriminates} the variety $\varr \D$ generated by $\D$, that is, $\D \sub \varr \D$ and for
every finite set $V$ of words in, say, $n$ variables, none of which is identically $1$ in $\varr
\D$ there is a group $D\in \D$  and elements $d_1,d_2,\ldots,d_n\in D$ such that for all $w\in
V$ $w(d_1,d_2,\ldots,d_n)\not=1$ holds~\cite{B+3N}. Discriminating set $\D$ always generates
$\varr \D$. If a discriminating set consists of one group $D$ we term {\it discriminating
group} $D$. 

\begin{lemma}[see~\cite{B+3N}]
\label{D<QS(D_1)}
If $\D$ discriminates the variety $\U = \varr{\D}$ and for the given set $\D_1$ the relations $\D \sub {\sf Q S} \D_1$ and $\D_1\sub \U$ hold, then $\D_1$ also discriminates $\U$. 
\end{lemma}

Now we can prove the following:

\begin{lemma}
\label{InfiniteExpB}
If the group $B$ is not of finite exponent, then the cartesian wreath product $A\Wr B$ generates 
the variety $\var{A}\cdot \var{B} = \var{A}\cdot \A$. 
\end{lemma}

\begin{proof}
Assume, firstly, that $B$ contains an element $c$ of infinite order. Then $B$ contains an 
infinite cycle $C=\langle c \rangle$ which is a discriminating group for the variety $\A$~\cite{B+3N}. According to Lemma~\ref{D<QS(D_1)}, $B$ also discriminates $\A$. Therefore
for an {\it arbitrary} (and not only abelian) group $A$ the wreath product $A\Wr B$ discriminates
$\var A \cdot \A$ because in this situation the {\it direct} wreath product $A\Wrr B$ discriminates
$\var A \cdot \A$~\cite{B+3N}: $A\Wrr B \in {\sf S}(A\Wr B)$.

Assume now that $B$ is not of finite exponent but that $B$ contains no element of infinite order.
Then there exists a sequence of elements
\begin{equation}
\label{ccc}
	c_1,c_2,\ldots \in B 
\end{equation}
such that for arbitrary $l\in \N$ there is such a $c_{i(l)}$ whose $\exp{c_{i(l)}} \ge l$. Let 
\begin{equation}
\label{w(x)...w(x)}
	w_1(x_1,\ldots ,x_d),\ldots,w_k(x_1,\ldots ,x_d) 
\end{equation}
be a finite set of words none of which is an identically $1$ for all abelian groups (we can 
assume all these words to contain the same variables $x_1,\ldots,x_d$ because the number of
these words is finite). Since the infinite cycle $C=\langle c \rangle$ discriminates $\A$, there exist elements
$c^{i_1}, \ldots ,c^{i_d}\in C$  ($i_1,\ldots,i_d \in \Z$) such that 
\begin{equation}
\label{w(c)...w(c)}
w_1(c^{i_1},\ldots ,c^{i_d})\not= 1,\ldots,w_k(c^{i_1},\ldots ,c^{i_d}) \not= 1.
\end{equation}
Now let us consider all these values~(\ref{w(c)...w(c)}) together with all values of all {\it 
subwords} of words~(\ref{w(x)...w(x)}) over elements $c^{i_1}, \ldots ,c^{i_d}$. Since in this
way we will get only \textit{finitely many} values, that is, only finitely many elements $c^j\in C$, we
can choose a power $c^{j_0}$, such that the absolute value of $j_0$  is greater than that of
all $j$'s obtained.  Now take such a $c_{i(j_0)}$ in~(\ref{ccc}) that $\exp{c_{i(j_0)}}$ is greater than $\abs{j_0}$. Since, clearly, all values~(\ref{w(c)...w(c)}) remain unchanged if they are
calculated not in $C$ but in $\langle c_{i(j_0)} \rangle$ (that is, modulo
$\exp{c_{i({j_0})}}$), we get that cycles
\begin{equation}
\label{<c><c>}
\langle c_{1} \rangle, \langle c_{2} \rangle, \ldots \le B
\end{equation}
form a discriminating set for the variety $\A$. And since $\{ \langle c_{1} \rangle, \langle c_{1} 
\rangle, \ldots\} \sub {\sf S} B$, the group $B$ discriminates $\A$ according to
Lemma~\ref{D<QS(D_1)}.
\end{proof}

An analog of this lemma for the ``passive'' group $A$ is also true:

\begin{lemma}
\label{InfiniteExpA}
If the group $A$ is not of finite exponent, then cartesian wreath product $A\Wr B$ generates 
the variety $\var{A}\cdot \var{B} = \A\cdot \var{B}$. 
\end{lemma}
\begin{proof}
Taking into account Lemma~\ref{InfiniteExpB} we assume, without loss of generality, that $B$ is 
a group of finite exponent $n$. Since the infinite cycle $C$ belongs to $\var A = \A$, it
sufficient, according to Lemma~\ref{karevor}, to prove that 
$$
	\var{C \Wr C_{n}}=\var{A}\cdot \var{B}=\A \cdot \A_{n},
$$
where $C_{n}$ is a finite cycle of order $n$. We can choose infinitely many cycles 
$$
	C_{p_1}, C_{p_2}, \ldots \in {\sf Q} (C)
$$ 
of orders $p_1,p_2,\ldots$ all coprime to $n$ and to each other. Thus $\var{C \Wr C_{n}}$ 
contains each one of the wreath products 
$
	C_{p_1} \Wr C_{n},\,\, C_{p_2} \Wr C_{n}, \ldots
$
and, thus, varieties $\A_{p_1}\cdot \A_{n}, \,\, \A_{p_2}\cdot \A_{n}$ generated by these 
wreath products according to the result of Houghton~\cite{HannaNeumann}.  It remains to use the
fact that:  
$$
	(\A_{p_1}\cdot \A_{n}) \cup (\A_{p_2}\cdot \A_{n}) \cup \cdots =
	\left( \A_{p_1}\cup \A_{p_2} \cup \cdots \right)\cdot \A_{n}
	=\A \cdot \A_{n}.
$$
\end{proof}

We collect the information of Lemmas~\ref{InfiniteExpB} and~\ref{InfiniteExpA} below:

\begin{theorem}
\label{InfiniteExpAorB}
If at least one of the groups $A$ and $B$ is not of finite exponent, then cartesian or direct 
wreath product of groups $A$ and $B$ generates the variety $\var{A}\cdot \var{B}$. 
\end{theorem}

\section{The case of finitely generated abelian groups}
\label{The_case_of_finitely_generated_groups}

\noindent
Assume $A$ and $B$ to be arbitrary {\it finitely generated}\, abelian groups. If in a direct 
decomposition of $A$ or $B$ an infinite cycle is present, then we apply the construction of
Section~\ref{The_cases_of_non-finite_exponent}. So what we have to deal with are merely {\it
finite} abelian groups $A$ and $B$.

\begin{lemma}
\label{FiniteAB}
For finite abelian groups $A$ and $B$ of exponents $m$ and $n$ respectively the wreath product 
of $A$ and $B$ generates the variety $\var{A}\cdot \var{B}=\A_{m} \cdot \A_{n}$ if and only if the exponents $m$ and $n$ are coprime.
\end{lemma}

\begin{remark}
\label{remark1}
So, as we see, {\it the results of Higman~\cite{Some_remarks_on_varieties} and 
Houghton~\cite{HannaNeumann} remain true not only for finite cycles, but also for arbitrary
finite groups}. On the other hand, as the familiar example $\var{C_p \Wr \prod_{i=1}^{\infty}
C_p}= \A_p \cdot \A_p$ shows ($p$ is a prime), the criterion of Higman  and Houghton has no
direct analog for the case of infinite groups, even for the case of infinite groups of finite
exponent.
\end{remark}

\begin{proof}[Proof of Lemma~\ref{FiniteAB}]
If $\var{A}\cdot \var{B}=\A_{m} \cdot \A_{n}= \var{A\Wr B}$, then $\A_{m} \cdot \A_{n}$ is a 
Cross variety~\cite{HannaNeumann} and $m$ and $n$ are 
coprime according to result of {\v S}melkin on product varieties of group generated by a finite group~\cite{ShmelkinOnCrossVarieties}.

On the other hand, if the condition of the lemma is satisfied, we can choose a cyclic subgroup 
$\langle a_{m} \rangle$ of order $m$ in $A$ and a cyclic subgroup $\langle b_{n} \rangle$ of
order $n$ in $B$. Now according to Lemma~\ref{karevor} $\var{A \Wr B}=\var{\langle a_{m}
\rangle \Wr \langle b_{n} \rangle}$ and the latter is equal $\A_m\cdot \A_n$ according to the
result of Houghton.
\end{proof}

Let us present the information of Sections~\ref{The_cases_of_non-finite_exponent} 
and~\ref{The_case_of_finitely_generated_groups} in an ``easy-to-use'' form:

\begin{theorem}
\label{FinitelyGenerated}
Let $A$ and $B$ be arbitrary finitely generated abelian groups with direct decompositions 
respectively
$$	A= \underbrace{C \oplus  \cdots \oplus C}_{r_A} \oplus \,
	C_{p_1^{u_1}} \oplus \cdots \oplus C_{p_{s}^{u_s}}
\,
\mathrm{and}
\,\,\,
	B= \underbrace{C \oplus  \cdots \oplus C}_{r_B} \oplus \,
	C_{q_1^{k_1}} \oplus \cdots \oplus C_{q_{d}^{k_{d}}}.
$$
Then the cartesian or direct wreath product of groups $A$ and $B$ generates the product variety 
$\var A \cdot \var B$ if and only if
\begin{enumerate}
\item at least one of the decompositions of $A$ and $B$ contains  a non trivial infinite cycle, 
that is, $r_A\not=0$ or $r_B\not=0$,

\item or all cycles in the decompositions of $A$ and $B$ are finite, that is, $r_A =0$, $r_B =0$,
 and the sets of  primes $\{p_1,\ldots,p_{s}\}$ and $\{q_1,\ldots,q_{d}\}$ have empty intersection.
\end{enumerate}
\end{theorem}

\section{The case of arbitrary abelian $p$-groups}  
\label{The_case_of_abelian_$p$-groups}

\subsection{Some notations}
As we mentioned in Remark~\ref{remark1}, in the case of infinitely generated abelian groups $A$ 
and $B$ of finite exponent no analogs of Lemma~\ref{FiniteAB} and of
Theorem~\ref{FinitelyGenerated} do exist. What a variety can be generated by wreath product,
say, 
$
	C_p \Wr (C_{p^2}\oplus \sum_{i=1}^{\infty}C_p)
$
of groups of exponent $p$ and $p^2$? We clear this situation by means of a specially defined 
function $\lambda(A,B,t)$:

\begin{definition}
\label{lambda(A,B,t)}
For given abelian $p$-groups $A$ and $B$ of finite exponents and for given $t\in \N$ the value of the function
$$
	\lambda = \lambda(A,B,t), 
$$
is defined to be the {\it maximum of the nilpotency classes of the $t$-generated groups} of variety 
$\var{A\Wr B}$. 
\end{definition}

Firstly we have to observe that this definition is {\it correct}: $t$-generated groups of 
$\var{A\Wr B}$ belong to the variety generated by all $t$-generated subgroups of the group
$A\Wr B$~\cite[16.31]{HannaNeumann}. The number of non-isomorphic copies of mentioned
$t$-generated subgroups is finite and every one of them is {\it a finite $p$-group} and, thus,
nilpotent (in the next subsection we will find concrete upper bounds for
nilpotency classes of $t$-generated groups of $\var{A\Wr B}$).

For the purposes of the rest of this paper for the given $p$-group $B$ of finite exponent $p^k$ 
let us denote: $k(B,p)=k$. If $B$ is not a $p$-group but still has finite exponent, then let
us denote by $k(B,p)$ the largest $k$ for which  $p^k$ divides $\exp B$. Further, for the
given $B$ and positive integer $s$ let us denote 
$
	B[s]=\langle b\in B \,|\, b^s=1 \rangle.
$ 
The factor groups 
$$
B[p^k]/B[p^{k-1}]=B[p^{k(B,p)}]/B[p^{k(B,p)-1}]
$$
will play a key role in our construction. So let us clear what do they mean in our situation of 
$p$-groups. According to Pr\"ufer's Theorem the group $B$ is a direct sum of finite cycles which
can be arranged as:
\begin{equation}
\label{TheFormOfB}
	B= C_{p^{k_1}} \oplus \,\,  C_{p^{k_2}} \oplus \cdots,
\end{equation}
where $k_1=k=k(B,p) \ge k_2 \ge \cdots$. From this decomposition it is clear that 
\begin{equation}
\label{TheFormOfB[]/B[]}
	B[p^{k}]/B[p^{k-1}]=\underbrace{C_{p} \oplus
	\cdots  \oplus \,\, C_{p}}_{\text{$\mu$ times}},
\end{equation}
where $\mu$ is such an ordinal that $k_1=k, \ldots ,k_\mu=k $ and 
$k_{\mu +1}<k$ (notation is correct for the set of all ordinals greater than $\mu$ can be 
well-ordered).

\subsection{An upper bound for the function $\lambda(A,B,t)$}
\label{An_upper_bound_for_the_function}
According to the remark following Definition~\ref{lambda(A,B,t)}, $\lambda(A,B,t)$ is bounded by the
maximum of the nilpotency classes of $t$-generated subgroups of $A\Wr B$. Let $H$ be such a
subgroup and $A_H$ be its intersection with the base group $A^B$ of $A\Wr B$. So $A_H$ is a group of the variety $\A_{p^u}$, where $p^u=\exp A$. $A_H$ is normal in $H$
and $H/A_H$ is isomorphic to an, at most, $t$-generated subgroup $A_B$ of 
$$
	(A\Wr B)/A^B \cong A
$$
and, therefore, according to Kaloujnine and Krasner Theorem~\cite{KaloujnineKrasner}, we get, 
that $H$ is embeddable into the  cartesian wreath product $A_H \Wr B_H$, where $A_H$ is of
exponent $p^{u'}$ dividing $p^u$ and where 
$$
	B_H = C_{k'_1} \oplus C_{k'_2} \oplus \cdots \oplus C_{k'_t} \quad (k'_1\ge k'_2 \ge \cdots\ge 
k'_t)
$$ 
is certain subgroup of $B$. So $B_H$ is finite and $A_H \Wr B_H$ is a {\it direct} wreath product. 
Thus, applying the result of Liebeck~\cite{Liebeck_Nilpotency_classes} (on the nilpotency class of
the direct wreath product of finite abelian $p$-groups), we calculate the class $c$ of $A_H \Wr
B_H$:
$$
	c=\sum_{i=1}^t (p^{k'_i}-1)+ (u'-1)(p-1)p^{k'_1-1}+1.
$$
And since $u'\le u$, \,\,$k'_1 \le k_1 = k$,\,\, $k'_2\le k_2$, $\ldots$, we have:
\begin{lemma}
For given abelian $p$-group $A$ of finite exponent $p^u$ and for the abelian $p$-group $B$ of form~(\ref{TheFormOfB}):
\begin{equation}
\label{ValueForLambda}
	\lambda(A,B,t) \le \sum_{i=1}^t (p^{k_i}-1)+ (u-1)(p-1)p^{k(B,p)-1}+1.
\end{equation}
\end{lemma}

\subsection{An example of a $t$-generated group in $\A_{p^u}\cdot \A_{p^k}$, the general 
criterion for wreath products of abelian $p$-groups}

\begin{example}
The product variety $\A_{p^u}\cdot\A_{p^k}$ contains the $t$-generated group $
	T(p,t)=C_{p^u}\Wrr \sum_{i=1}^{t-1} C_{p^k}.
$
According to~\cite{Liebeck_Nilpotency_classes} the nilpotency class of $T(p,t)$ is equal to:
$$
	\nu(p,t)= \sum_{i=1}^{t-1} (p^{k}-1)+ (u-1)(p-1)p^{k-1}+1.
$$
\end{example}

\begin{lemma} 
\label{nu>mu}
If $\abs{B[p^k]/B[p^{k-1}]}<\infty$, then for sufficiently large values of $t$:
$$
	\nu(p,t)> \lambda(A,B,t)
$$
and, thus, the group $T(p,t)$ does not belong to the variety $\var{A\Wr B}$.

On the other hand, if $\abs{B[p^k]/B[p^{k-1}]}=\infty$, then 
$\var{A\Wr B}=\A_{p^u}\cdot\A_{p^k}$.
\end{lemma}

\begin{proof}
Assume $B[p^k]/B[p^{k-1}]$ is finite and is of order $p^{\,\mu}$\!, according to~(\ref{TheFormOfB[]/B[]}). If $t\le\mu$, then all powers of $p$ on the right side of~(\ref{ValueForLambda}) are equal to $p^k$. But when $t$ becomes greater than $\mu$, then the
later summands in sum $\sum_{i=1}^t (p^{k_i}-1)$ become less or equal to $p^{k-1}-1$. Thus for
sufficiently large $t$:
\begin{align*}
	\nu(p,t)- \lambda(A,B,t) &> \sum_{i=1}^{t-1}(p^k-1) 
	-\left( \sum_{i=1}^{\mu}(p^k-1)+ \sum_{i=\mu +1}^{t}(p^{k-1}-1)\right)\\
	&= \sum_{i=\mu +1}^{t-1}\left[(p^k-1)-(p^{k-1}-1)\right] - (p^{k-1}-1)\\
	&= (t -1 -\mu)(p^k-p^{k-1})+1 - p^{k-1}.
\end{align*}
So taking an arbitrary positive integer $t_0> (p^{k-1}-1)/(p^k-p^{k-1})+\mu +1$ we get that 
$	\nu(p,t_0)- \lambda(A,B,t_0) > 0$.

\vskip2mm
Assume now $B[p^k]/B[p^{k-1}]$ is infinite, that is, $\mu$ is an infinite ordinal 
in~(\ref{TheFormOfB[]/B[]}). Thus $B$ contains a subgroup $D$ isomorphic with an infinite direct
power of cycle $C_{p^k}$. $D$ is a discriminating group in the variety $\A_{p^k}$~\cite{B+3N}.
Therefore by Lemma~\ref{D<QS(D_1)} the group $B$ itself is a discriminating group for the variety 
$\A_{p^k}$ and so $\var{A \Wr B}=\A_{p^u}\cdot \A_{p^k}$.
\end{proof}
Now let us summarise:

\begin{theorem}
\label{p-groups}
Let $A$ and $B$ be arbitrary abelian $p$-groups. Then:
\begin{enumerate}
\item if at least one of the groups $A$ and $B$ is not of finite exponent, then $\var{A\Wr B}=\var{A\Wrr B}=\var{A} \cdot \var{B}$,
\item if $A$ and $B$ are groups of finite exponents $p^u$ and $p^k$ respectively
then\hfil\break 
$\var{A\Wr B}=\var{A\Wrr B}=\var{A} \cdot \var{B}=\A_{p^u}\cdot \A_{p^k}$ if and only if in
direct decomposition~(\ref{TheFormOfB}) of $B$ infinitely many cycles of order $p^k$ are
present, that is, if the factor group $B[p^k]/B[p^{k-1}]$ is infinite.
\end{enumerate}
\end{theorem}

\begin{example}
\label{nice_example}
Applying this result, we easily find an answer to the question asked at the beginning of this 
section: the group 
$
C_p \Wr (C_{p^2}\oplus \sum_{i=1}^{\infty}C_p)
$
does {\it not} generate the variety $\A_{p}\cdot \A_{p^2}$ because decomposition of the 
\lq\lq active\rq\rq\  group of this wreath product contains only one direct summand of order
$p^2=p^k$.
\end{example}

\section{The case of abelian groups of finite composite exponents}  
\label{groups_of_finite_composite_exponents}

\subsection{$p$-groups in variety $\var{A\Wr B}$}
Assume $\exp A =m$, $\exp B =n$ and denote for a given prime $p$ (not necessarily dividing both 
$m$ and $n$) by $A_{p}$ and $B_{p}$ the $p$-primary components of $A$ and $B$ respectively. 

\begin{lemma}
\label{TechnivalResult}
Using this notation:
\begin{equation}
\label{TechnivalResultsFormule}
\var{A \Wr B}\,\cap\, \A_{p^u} \cdot \A_{p^k}= \var{A_p \Wr B_p},
\end{equation}
where $u=k(A,p)$, $k=k(B,p)$.
\end{lemma}

\begin{proof}
The right side of~(\ref{TechnivalResult}) lies in the left side. So it is sufficient to prove 
that every $p$-group  $P$ of  $\var{A \Wr B}$ belongs to $\var{A_p \Wr B_p}$. Moreover, since
$\var{A \Wr B}$ is a locally finite variety, we can assume $P$ to be a {\it finite}
group~\cite{HannaNeumann}. We omit the trivial case, when $p$ is coprime with $m$ or with $n$,
that is, when $A_p=\1$ or $B_p=\1$. The variety $\var{A \Wr B}$ is generated by the set $\{ R_i
\,|\,i\in I \}$ of all finite subgroups of $A \Wr B$. So there is a finite subset
$\{R_1,\ldots, R_l \}$ of this set, such that $P\in \var{R_1,\ldots, R_l}$~\cite{HannaNeumann}
and:
$$
	P\in {\sf QSC}(R_1,\ldots, R_l).
$$
That is, $P$ is a surjective image of a subgroup $R$ of a direct product $R_1 \times \ldots \times 
R_l$ under some homomorphism $\varphi: R\to P$:
$$
	P=\varphi (R), \quad R \le R_1 \times \ldots \times R_l.
$$
$P$ is a $p$-group and, thus, is an image of {\it some} Sylow $p$-subgroup $P^*$ of $R$. In 
turn $P^*$ is a sub-direct product of its projections $P^*_i$ on $R_i$,\,\,
$i=1,\ldots,l$. Denote by $P^*_{i,B}$ the intersection of $P^*_{i}$ with the base group $A^B$
of $A\Wr B$. Clearly, $P^*_{i,B}$ is normal in $P^*_{i}$ and the factor group
$P^*_{i}/P^*_{i,B}$ is isomorphic to some subgroup $P^*_{i,A}$ of the factor group $(A\Wr B)/A^B
\cong A$. Thus $P^*_i$ is isomorphically embeddable into the wreath product $P^*_{i,A} \Wr
P^*_{i,B}$. The group $P^*_{i,B}$ is a $p$-subgroup of $B$ and, thus, lies in the $B_p$.
$P^*_{i,A}$ is a $p$-subgroup of the base group $A^B$ and, thus, lies in cartesian power
$(A_p)^B$ of $A_p$. So $P^*_{i,A}\in \var{A_p}$ and, according to Lemma~\ref{karevor}, 
$$
P^*_i \le P^*_{i,A} \Wr P^*_{i,B} \in \var{{A_p}\Wr {B_p}}.
$$
\end{proof}

\subsection{Wreath products of abelian groups of finite composite exponents}

Assume $\exp A =m$, $\exp B =n$ as above and
$$
A=A_{p_1}\oplus \cdots \oplus A_{p_s}, \quad 
B=B_{q_1}\oplus \cdots \oplus B_{q_d}
$$
are direct decompositions of $A$ and $B$ as direct sums of finitely many primary components 
$A_{p_i}$, $i=1,\ldots,s$ \,\, and\,\, $B_{p_j}$, $j=1,\ldots,d$.

Let us begin with the case when the set $\{ q_1,\ldots,q_d\}$ is a subset of 
$\{ p_1,\ldots,p_s\}$. Assume $p_1=q_1$ and denote for brevity $p=p_1=q_1$. Assume further that
$\exp{A_p}=p^u$ and $\exp{B_p}=p^k$.

\begin{lemma}
\label{TheCasem|n}
If $A$, $B$, $p$ are as above, then
$$
	\var{A\Wr B}= \var{A} \cdot \var{B} = \A_{m}\cdot \A_{n}
$$
if and only if for each $p$ dividing $n$ the factor $B[p^k]/B[p^{k-1}]$, where $k=k(B,p)$, is 
infinite.
\end{lemma}

\begin{proof}
If $A\Wr B$ generates $\A_m \cdot \A_n$, then according to Lemma~\ref{TechnivalResult}:
$$
	\var{A_p \Wr B_p} = \var{A \Wr B} \cap  \, \A_{p^u} \cdot \A_{p^k} = \A_{p^u} \cdot \A_{p^k}.
$$
Therefore the group $B_p$ must satisfy the condition of Lemma~\ref{nu>mu}:
$$
	\abs{B_p[p^{k(B_p,p)}]/B_p[p^{k(B_p,p)-1}]}=\infty. 
$$
But, since $B_p$ is the $p$-primary component of $B$, we have:
$$
	k(B_p,p)=k(B,p)=k \quad {\rm and} \quad B_p[p^{k(B_p,p)}]=B_p[p^{k}]=B[p^{k}].
$$
\vskip2mm
Now assume, on the other hand, that for all $q_1,\ldots,q_d$ all factors 
$$
B[q_1^{k_1}]/B[q_1^{{k_1}-1}]\, ,\ldots, 
B[q_1^{k_d}]/B[q_d^{{k_d}-1}]\, ,
$$
where $k_i=k(B_{q_i}, {q_i})$ \,\,($i=1,\ldots,d$)
are infinite. Since the cycle $C_n$ is the direct sum of cycles $C_{q^{k_1}},\ldots,C_{q^{k_d}}$, 
we get that $B$ contains the infinite direct power $D$ of cycle $C_n$. And since $D$
discriminates $\A_n$, the group $B$ also  discriminates $\A_n$ and $A\Wr B$ discriminates
$\A_m\cdot\A_n$.
\end{proof}

It remains to consider the case when the exponent $n$ of $B$ has a prime divisor $p$ which does
 {\it not}\, divide the exponent $m$ of $A$. Let 
$$
	m=p_1^{u_1}\cdots p_s^{u_s}
	\quad {\rm and} \quad
	n=q_1^{k_1}\cdots q_{s'}^{k_{s'}} q_{s'+1}^{k_{s'+1}}\cdots 	q_{d}^{k_{d}},
$$
where primes $p_i$ and $q_j$ are arranged such that $p_1=q_1, \ldots , \,\, p_{s'}=q_{s'}$ and 
$p_i\not= q_j$ for all $i=1,\ldots,s$;\,\, $j=s'+1,\ldots,d$. Then $B$ has a decomposition $B_1
\oplus B_2$, where $n_1=\exp{B_1}=q_1^{k_1}\cdots q_{s'}^{k_{s'}}$ and
$n_2=\exp{B_2}=q_{s'+1}^{k_{s'+1}}\cdots q_{d}^{k_{d}}$.

\begin{lemma}
In the above notation the wreath product $A \Wr B$ generates the variety $\A_m\cdot\A_n$ if and only 
if the wreath product $A \Wr B_1$ generates the variety $\A_m\cdot\A_{n_1}=\A_m\cdot \var{B_1}$.
\end{lemma}

\begin{proof}
If $\var{A \Wr B}=\A_m\cdot\A_n$, then $\var{A \Wr B_1}=\A_m\cdot\A_{n_1}$ according to 
Lemma~\ref{TechnivalResult} and to the first part of the proof of Lemma~\ref{TheCasem|n}.

\vskip2mm
Assume, on the other hand, $\A_m\cdot\A_{n_1}=\var{A \Wr B_1}$. Then $B_1$ contains the 
direct sum $B_1^*$ of infinitely many copies of $C_{q_i^{k_i}}$ for each $i=1,\ldots,s'$:
$$
	B_1^*=\sum_{j=1}^{\infty}C_{q_1^{k_1}} \,\oplus \cdots \oplus\, 
	\sum_{j=1}^{\infty}C_{q_{s'}^{k_{s'}}} \le B_1.
$$
The groups $B_2$ contains the cycle
$$
	B_2^*= C_{n_2}\cong C_{q_{s'+1}^{k_{s'+1}}}
	\!\oplus \cdots \oplus C_{q_{d}^{k_{d}}}.
$$
Thus, according to Lemma~\ref{karevor}, it is sufficient to prove that 
$\A_m \cdot \A_n = \var{A\Wr B^*}$, where $B^*=B_1^* \oplus B_2^*$.

$\A_m \cdot \A_n$ is a locally finite variety generated by its critical 
groups~\cite{HannaNeumann}. Let $Q$ be such a critical group. $Q$ is an extension of a normal
subgroup $H\in \A_m$  by means of group $G=Q/H \in \A_n$. So $G=G_1 \oplus G_2$, where $G_1 \in
\A_{n_1}$ and $G_2 \in \A_{n_2}$. Let $M$ be the monolith of $Q$~\cite{HannaNeumann}. $M$ is a
finite direct product of, say, $r$ copies of some cycle $C_p$, $p\in \{ p_1,\ldots, p_s\}$. So
$M$ is an  $r$-dimensional space over the finite field $\F_p$ and the operation of conjugation of
elements of $M$ by elements of $Q$ defines a linear representation of the group $G$ degree $r$ over the field
$\F_p$. Since $\exp{G_2}$ is coprime with $p$ we think of this groups to be isomorphically embedded in $Q$. Let us use the same notaion $G_2$ for that isomorphic copy in $Q$.
Since $M$ is a {\it minimal} normal subgroup of $Q$, this representation of $G$ is irreducible. Let us
apply the theorem of Clifford~\cite{Curtis_Reiner} to the representation of {\it normal
subgroup} $G_2$ of $G$, that is, to $\F_p G_2$-module $M_{G_2}$. Our $\F_p G$-module $M$ is, thus, a
direct sum of its submodules, each of which is homogenous regarding $G_2$:
$M=M_1 \oplus \cdots  \oplus M_l$. Since $H$ and $Q/H$ are both abelian, $M_1,\ldots, M_l$ are 
normal in $Q$. Thus $l=1$ and $M=M_1$.

The representation of $G_2$ is {\it faithful}. For, if not, the direct sum $H \oplus K$ of $H$ 
and of non-trivial kernel $K$ would be normal in $Q$ and, so, $D$ would contain a ``second
monolith'' of $Q$. Thus, as an abelian group with irreducible faithful representation, $G_2$
has to be a {\it cycle}. Since this cycle belongs to $\A_{n_2}$, we get that $G_2$ is a
subgroup of $B_2^*$. Further, since $G_1$ is a finite group in $\A_{n_1}$, then $G_1$ is a
subgroup of $B_2^*$. Therefore $G=G_1 \oplus G_2$ is a subgroup of $B_1^* \oplus B_2^*$ and,
thus, the extension $Q$ of $H$ by $G$ lies in $H\Wr G$ and 
$$
	H\Wr G \le H\Wr B^* \in \var{A\Wr B^*}.
$$
\end{proof}

\begin{example}
\label{extended_nice_example}
Replacing the wreath product of Example~\ref{nice_example} by the following one:
$
W=C_p \Wr (C_{p^2}\oplus \sum_{i=1}^{\infty}C_p  \oplus  \sum_{i\in I}C_q)
$, where $q$ is a prime different from $p$, we get that $W$ does {\it not}\, generate the variety 
$\A_p \cdot \A_{p^2\cdot q}$ for any index set $I$, in spite of the fact that $C_p \Wr
\sum_{i\in I}C_q$ does generate the variety $\A_p \cdot \A_q$ for arbitrary non-empty index set
$I$. On the other hand, replacing in $W$ the summand $C_{p^2}$ by $\sum_{i=1}^{\infty}C_{p^2}$,
we will obtain a wreath product $C_p \Wr \left[\sum_{i=1}^{\infty}C_{p^2} \oplus
\sum_{i=1}^{\infty}C_{p}\oplus \sum_{i\in I}C_q\right]$ {\it generating} the variety $\A_p \cdot
\A_{p^2\cdot q}$ for any non-empty $I$. Clearly this wreath product will generate variety $\A_p
\cdot \A_{p^2\cdot q}$ even if we remove the summand $\sum_{i=1}^{\infty}C_{p}$ in the active
group.
\end{example}

The information of this section can be collected as:

\begin{theorem} 
\label{ArbitraryFinitem,n}
Let $A$ and $B$ be arbitrary abelian groups of finite exponents $m$ and $n$ respectively and 
let 
$
	m=p_1^{u_1}\cdots p_s^{u_s} \,\, {\rm and} \,\, n=q_1^{k_1}\cdots 	q_{s'}^{k_{s'}} 
q_{s'+1}^{k_{s'+1}}\cdots q_{d}^{k_{d}},
$
where $s'\le s$ and where the prime divisors $p_i$ and $q_j$ are grouped such that 
$p_1=q_1$,...,\,\,$p_{s'}=q_{s'}$ and $p_i\not= q_j$ for all $i=1,\ldots,s$;\,\,
$j=s'+1,\ldots,d$. Then 
$
\var{A\Wr B}=\var{A\Wrr B}=\var{A}\cdot \var{B} =\A_m \cdot \A_n
$ 
if and only if the factors $B[q^{k(B,q)}]/B[q^{k(B,q)-1}]$ are infinite for all 
$q=q_1,\ldots,q_{s'}$.
\end{theorem}

\section{The general case of arbitrary abelian groups\\ and wreath products of groups
\lq\lq near\rq\rq\  to abelian ones}

\subsection{The general criterion for arbitrary abelian groups}

Statements of Theorems~\ref{InfiniteExpAorB}, \ref{FinitelyGenerated}, \ref{p-groups} 
and ~\ref{ArbitraryFinitem,n} are constituent parts of the following main criterion for
cartesian and direct wreath products of arbitrary abelian groups:

\begin{theorem}[{\sc Main Criterion}]
\label{MainCriterion}
For arbitrary abelian groups $A$ and $B$ their cartesian wreath product $A \Wr B$ (or direct 
wreath product $A \Wrr B$) generates the product variety $\var A \cdot \var B$ if and only if
\begin{enumerate}
\item at least one of the groups $A$ and $B$ is not of finite exponent,
\item or if $A$ and $B$ are of finite exponents $m$ and $n$ respectively and for each prime $p$ 
dividing both $m$ and $n$ the factors $B[p^k]/B[p^{k-1}]$ are infinite,
where $B[s]=\langle b\in B\,|\,b^{s}=1 \rangle$ and $p^k$ is the highest power of $p$ 
dividing~$n$.
\end{enumerate}
\end{theorem}

As the mentioned Theorems~\ref{InfiniteExpAorB}, \ref{FinitelyGenerated}, \ref{p-groups} 
and ~\ref{ArbitraryFinitem,n} show, this criterion is 
{\it effective} in the sense that in concrete cases the factors
$B[p^k]/B[p^{k-1}]\cong
\sum_{i\in I}C_{p}$ (if need of their consideration arises) have simple and
\lq\lq understandable\rq\rq\  meaning: we have to consider them only in the case when $A$ and $B$ are of finite exponents and there is a prime $p$ dividing that exponents; in this circumstances the condition $|B[p^k]/B[p^{k-1}]|=\infty$ simply means that in the direct decomposition $G_p=C_{p^{k_1}}\oplus C_{p^{k_2}}\oplus \cdots $ (where $k_1 \ge k_2 \ge \cdots$) of the $p$-primary component $B_p$ of the group $B$ infinitely many cycles (direct summands) $C_{p^{k_1}}=C_{p^{k(B,p)}}$ are present.

An immediate consequence of Theorem~\ref{MainCriterion} is the following:

\begin{corollary}
If for abelian groups $A$ and $B$ the wreath product $A\Wr B$ does not generate the product variety  $\var A \cdot \var B$, then the wreath product $A\Wr \sum_{i=1}^l B$ also does not generate $\var A \cdot \var B$ for any positive integer $l \in \N$.
\end{corollary}

\subsection{Parallels for nilpotent groups of class $2$ and metabelian groups. Problems}

The following two examples show, that Theorem~\ref{MainCriterion} does {\it not}\, have obvious
 generalizations even for cases of ``small'' groups of classes of groups ``near'' to abelian groups.

\begin{example}
Let $\V=\Ni_2 \cap \B_3$ be the variety of all nilpotent groups of class at most $2$ and 
exponent dividing $3$. Then for no one group $A$ generating $\V$ and cycle $B=C_2$ $\var{A \Wr B}=\V \cdot \A_2$. On the other hand $\exp A =3$ is coprime with $\exp B =2$.
According to Lemma~\ref{karevor}, it is sufficient to prove this for one group $A$ generating 
$\V$. Let
$
        A=F_2(\V) =\langle x_1,x_2 \,|\,        [x_1,x_2,x_1]=[x_1,x_2,x_2]=x_1^3=x_2^3=1 
\rangle
$
be the $\V$-free group of rank $2$ and let $R$ be the extension of $A$ by means of the group of 
operators generated by automorphisms $\nu_1,\nu_2\in \Aut{A}$ defined as:
$
        \nu_1 :x_1\mapsto x_1^{-1} ,\quad \nu_1 :x_2\mapsto x_2 ;\quad
        \nu_2 :x_1\mapsto x_1 ,\quad \nu_2 :x_2\mapsto x_2^{-1}.
$
Clearly $\langle \nu_1,\nu_2 \rangle \cong C_2 \oplus C_2 \in \A_2$. As it is shown 
in~\cite{BurnesDissertation}, $R$ is a critical group. Every one of its proper factors, but not
$R$ itself, satisfies
$
        [[x_1,x_2],[x_3,x_4],x_5]\equiv 1.
$
On the other hand the wreath product $A \Wr B$ satisfies this identity because its second commutator
 subgroup lies in the center.
\end{example}

\begin{example}
Let $A$ and $B$ be arbitrary finite groups generating varieties $\A_p \cdot \A_q$ and $\A_r$ 
respectively, where $p,q,r$ are arbitrary different primes. Then $\var{A \Wr B}\not=(\A_p \cdot
\A_q ) \cdot \A_r$, in spite of the fact that $\exp A =p\,q$ is coprime with $\exp B =r$.
For, the product of three non-trivial varieties $(\A_p \cdot \A_q ) \cdot \A_r=\A_p \cdot \A_q  
\cdot \A_r$ cannot, by theorem of \v Smelkin~\cite{ShmelkinOnCrossVarieties},  be generated by
a single finite group $A \Wr B$.
\end{example}

This examples, and the results of Section~\ref{QSC} proved not only for the wreath products of single groups but also for sets $\X \Wr\hskip0.1mm \Y$ set the following problems of generalization of our main criterion (Theorem~\ref{MainCriterion}) in two possible directions:

\begin{problem}
Let $A$ and $B$ be arbitrary (nilpotent, metabelian, soluble) groups. Find a criterion under which  $\var{A \Wr B} = \var A \cdot \var B$.
\end{problem}

\begin{problem}
Let $\X$ and $\Y$ be arbitrary sets of (abelian) groups. Find a criterion under which $\var{\X \Wr \Y} = \var{\X} \cdot \var{\Y}$.
\end{problem}

\section{Wreath products and finite direct sums of abelian groups}
\label{last}

As we saw in Section~\ref{QSC}, $\var{A \Wr B}=\var{A} \cdot \var{B}$ if and only if $\var{A 
\Wr B}$ contains wreath products $A\Wr (\prod_{i\in I}B)$ for every index set $I$. It is a result
of independent interest, that here instead of arbitrary set $I$ we can take a two-element set,
that is:

\begin{theorem}
\label{AnEvenStrongerResult}
For arbitrary abelian groups $A$ and $B$ equality
$
        \var{A \Wr B} = \var A \cdot \var B
$
takes place if and only the variety $\var{A \Wr B}$ contains the group $A \Wr (B \times B)$. 
The analogous statement is also true for direct wreath products of groups.
\end{theorem}

This theorem follows from a more general result:

\begin{theorem}
\label{MoreGeneralResult}
For arbitrary abelian groups $A$ and $B$ one and only one of the following two alternatives holds:
\begin{enumerate}
\item the variety $\var{A \Wr B}$ is equal to the variety $\var A \cdot \var B$ and, thus, to 
every subvariety $\var{A\Wr (\prod_{i\in I}B)}$ of the latter for any index set $I$;
\item the variety $\var{A \Wr B}$ is a proper subvariety of variety $\var A \cdot \var B$. Then 
for any positive integer $s\in \N$ the variety $\var{A \Wr (\prod_{i=1}^s B)}$ is a proper
subvariety of variety $\var A \cdot \var B$ and, moreover, of the variety $\var{A \Wr
(\prod_{i=1}^u B)}$ for any integer $u>s$. On the other hand for any infinite index set $I$: $\var{A\Wr (\prod_{i\in I}B)}=\var A \cdot \var B$.
\end{enumerate}

The first alternative holds for any abelian $A$ and $B$ apart from the following 
case: $A$ and $B$ are of finite exponents $m$ and $n$ respectively and for some prime $p$
dividing both $m$ and $n$ the corresponding factor $B[p^k]/B[p^{k-1}]$ is finite.

The analogous statement is also true for direct wreath products of groups.
\end{theorem}

\begin{remark}
Applying this theorem, having any wreath product $A \Wr B$ which does not generate the variety 
$\var A \cdot \var B$, we get a countably infinite set of linearly ordered proper subvarieties
of  $\var A \cdot \var B$, namely, varieties generated by wreath products:
$
        \{A \Wr (\underbrace{B\oplus \cdots \oplus B}_{\text{$s$ times}}) \,|   \, s\in \N \}
$
(see examples below).
\end{remark}

\begin{proof}[Proof of Theorem~\ref{MoreGeneralResult}]
First let us consider the cases, when the statement of this theorem easily follows from one of the
results already established. If one of the groups $A$ and $B$ is not of finite exponent, then the first alternative holds. Thus assume $\exp A =m$, $\exp B =n$. If the index set $I$ is infinite,
then for any $B$ the direct sum $\sum_{i \in I} B$ always discriminates $\var B$ and, thus,
${A\Wr (\sum_{i \in I} B)}$ discriminates $\var{A} \cdot \var{B}$. So assume $I$ to be finite.
If
$
        {\rm var}({A\Wr \sum_{i \in I} B})
        \not=\var{A} \cdot \var{B},
$
then by Theorem~\ref{MainCriterion} there is a prime $p$ such that the corresponding factor 
group $B[p^k]/B[p^{k-1}]$ is finite. Denote for brevity by $B_{s,p}$ the $p$-primary component
of direct sum $\sum_{i=1}^s B$. If for given $r>s$  $\var{A \Wr \sum_{i=1}^s B}=\var{A \Wr
\sum_{i=1}^r B}$ holds, then following the proof of Lemma~\ref{TechnivalResult}:
\begin{equation}
\label{A_pWrB_p}
        \var{A_p \Wr B_{s,p}}=\var{A_p \Wr B_{r,p}}.
\end{equation}
So to complete the proof it is sufficient to show that ~(\ref{A_pWrB_p}) leads to a 
contradiction to the fact that $B[p^k]/B[p^{k-1}]$ is finite. We can also omit the case when
the group $B_{s,p}$ (and, therefore, the group $B_{r,p}$) is finite, for in such a case by
Liebeck's  Theorem~\cite{Liebeck_Nilpotency_classes} the groups $A_p \Wr B_{s,p}$ and $A_p \Wr
B_{r,p}$ have different nilpotency classes.

For the rest of proof the concrete form of $B_p$ is essential. Since $B[p^k]/B[p^{k-1}]$ is 
finite, $B_p$ contains in its direct decomposition only {\it finitely many}, say $l_0$, 
summands $C_{p^k}$. It may turn out that $B[p^{k-1}]/B[p^{k-2}]$ is also finite and the number
of summands $C_{p^{k-1}}$ is finite, say $l_1$. But since the group $B_p$ is infinite, there
exists the {\it first} number $d$ such that the corresponding factor $B[p^{k-d+1}]/B[p^{k-d}]$
is finite, it consists of, say, $l_{d-1}$ summands $C_{p^{k-d+1}}$, and the factor
$B[p^{k-d}]/B[p^{k-d-1}]$ is {\it infinite}. Thus $B_p$ can be presented as:
\begin{equation}
\label{TheConcreteFormOfB_p}
\begin{split}
        B_p & = \underbrace{C_{p^k} \oplus \cdots  \oplus  C_{p^k}}_{l_0}
        \oplus \cdots  \oplus
        \underbrace{C_{p^{k-d+1}} \oplus \cdots  \oplus
        C_{p^{k-d+1}}}_{l_{d-1}} \\
        & \oplus \underbrace{C_{p^{k-d}} \oplus \cdots  \oplus
        C_{p^{k-d}}\oplus \cdots }_{\infty} \,\, \oplus  \,\, \hat{B},
\end{split}
\end{equation}
where $\exp{\hat{B}}\le p^{k-d-1}$ and where some of $l_1,\ldots,l_{d-1}$ may be equal to $0$. 
Let $\lambda = \lambda (A_p,B_{s,p},t)$ be the function defined in
Section~\ref{The_case_of_abelian_$p$-groups} and let $\exp{A_p}=p^u$. It follows from the
proof in Subsection~\ref{An_upper_bound_for_the_function} and from
decomposition~(\ref{TheConcreteFormOfB_p}) that, for any $t>s\cdot \sum_{i=0}^{d-1}l_i$ the function
$\lambda (A_p,B_{s,p},t)$ is bounded by
\begin{equation}
                s \cdot \sum_{i=0}^{d-1} l_i(p^{k-i}-1) +
                (t- s\cdot \sum_{i=0}^{d-1}l_i) \cdot(p^{k-d}-1) +
                (u-1)(p-1)p^{k-1} +1.
\end{equation}
On the other hand for each $t>r\cdot \sum_{i=0}^{d-1} l_i +1$ the variety $\var{A_p 
\Wr B_{r,p}}$  contains the following $t$-generated group:
$$
        T(r,t) = C_{p^u}\Wr \left[ \sum_{i=1}^{rl_0}C_{p^k}
        \oplus \cdots  \oplus  \sum_{i=1}^{rl_{d-1}}C_{p^{k-d+1}}
        \oplus  \sum_{i=1}^{\omega(t)}C_{p^{k-d}},
        \right]
$$
where
$\omega(t) = t- r \cdot \sum_{i=0}^{d-1} l_i -1$. The group $T(r,t)$ is nilpotent of class:
\begin{equation*}
\begin{split}
        \nu(p,r,t)= &
        r \cdot \sum_{i=0}^{d-1} l_i(p^{k-i}-1) +
                (t- r\cdot \sum_{i=0}^{d-1}l_i\,-1) \cdot(p^{k-d}-1)\\ +
        &       (u-1)(p-1)p^{k-1} +1.
\end{split}
\end{equation*}
It remains to verify that for sufficiently large integers $t$ the value of $\nu(p,r,t)$ is 
greater than that of $\lambda (A_p,B_{s,p},t)$. It is sufficient to make calculations for the
value $r=s+1$: the generality of the statement of Theorem~\ref{MoreGeneralResult} remains
unaffected but our calculations become much shorter.
\begin{equation*}
\begin{split}
        \nu(p,s+1,t) - \lambda (A_p,B_{s,p},t) & =
        \sum_{i=0}^{d-1} l_i(p^{k-i}-1) \\
        & + \left(t- (s+1)\cdot \sum_{i=0}^{d-1}l_i\,-1 -t + s\cdot
        \sum_{i=0}^{d-1}l_i \right) \cdot(p^{k-d}-1 )\\
        &= \sum_{i=0}^{d-1}l_i p^{k-i} - \sum_{i=0}^{d-1}l_i -
        \left(\sum_{i=0}^{d-1}l_i +1 \right)\cdot(p^{k-d}-1 )\\
        &= \sum_{i=0}^{d-1}l_i(p^{k-i}-p^{k-d})\,-p^{k-d} +1>0.
\end{split}
\end{equation*}
\vskip-3.5mm
\end{proof}
Here are two examples concerning subvarieties generated by wreath products in the lattice of 
all subvarieties of product varieties of abelian groups.

\begin{example}
Since $C_p \Wr C_p$ does not generate variety $\A_p^2$, wreath products $C_p \Wr \sum_{i=1}^s 
C_p$ \,\, ($s=1,2,\ldots$) \, generate infinitely many subvarieties of $\A_p^2$. We are able to
locate them in the lattice of subvarieties of $\A_p^2$ using its description due to Kov\'acs
and Newman~\cite{KovacsAndNewmanOnNon-Cross}. For any proper subvariety $\V$ of $\A_p^2$ there
is a number $s\ge 1$ such that
$
        {\rm var}\,({C_p \Wr \sum_{i=1}^{s} C_p}) \subseteq \V
        \subseteq  {\rm var}\,({C_p \Wr \sum_{i=1}^{s+1} C_p})
$
or $\V$ lies in $\var{C_p \Wr C_p}$. Moreover, for any $s\ge 1$  there are exactly the following 
$p-2$ subvarieties of $\A_p^2$ ``between'' $\var{C_p \Wr \sum_{i=1}^s C_p}$ and $\var{C_p \Wr
\sum_{i=1}^{s+1} C_p}$:
$
        \A_p^2 \cap \B_{p^2} \cap \, \Ni_{j},
        \, j=s(p-1)+2,\ldots, (s+1)(p-1).
$
And there are $2p-3$ subvarieties of $\A_p^2$  ``between'' $\var{C_p \Wr \1}=\A_p$ and 
$\var{C_p \Wr C_p}$ (see~\cite{KovacsAndNewmanOnNon-Cross}).
\end{example}

\begin{example}
On the other hand for arbitrary coprime numbers $m$ and $n$ the variety $\A_m\cdot \A_n$ 
contains only finitely many subvarieties~\cite{HannaNeumann}. According to
Theorem~\ref{MoreGeneralResult} this fact already guarantees, that for an arbitrary pair $A\in
\A_m$,  $B\in \A_n$: $\var{A\Wr B}=\var{A} \cdot \var{B}$.
\end{example}

Closing the current work we would like to announce our recent papers~\cite{SubnormalEmbeddingTheorems,Uber_die_normalen_verbalen_Einbettungen,SubnormalEmbedOfOrderedGr} as well as our common paper with Professor H.~Heineken~\cite{HeinekenMikaelianOnnVEmb}, where some related properties of wreath products and their verbal subgroups are considered.

\bibliographystyle{amsalpha}

\end{document}